\documentclass[reqno]{amsart}

\usepackage{amsmath, amssymb}
\usepackage[mathscr]{eucal}
\usepackage{upgreek}
\usepackage{hyperref}
\hypersetup{colorlinks=true, linkcolor=blue, urlcolor=blue, citecolor=[rgb]{0, 0.8, 0}}

\theoremstyle{plain}
\newtheorem*{theorem}{Theorem}
\newtheorem*{lemma}{Lemma}

\title{On the McLaughlin--Rundell theorem}

\author{Namig J. Guliyev}
\address{Institute of Mathematics and Mechanics, 9 B.~Vahabzadeh str., AZ1141, Baku, Azerbaijan.}
\address{Center for Mathematics and its Applications, Khazar University, 41 Mahsati str., AZ1096, Baku, Azerbaijan.}
\email{njguliyev@gmail.com}



\begin{document}
\begin{abstract}
  We give a one-sentence proof of McLaughlin and Rundell's inverse uniqueness theorem.
\end{abstract}
\maketitle

In their 1987 paper \cite{MR87}, McLaughlin and Rundell established a new kind of inverse uniqueness result. They considered the one-dimensional Schr\"{o}dinger equation
\begin{equation} \label{eq:SL}
  -y''(x) + q(x)y(x) = \lambda y(x)
\end{equation}
with the Dirichlet boundary condition at the left endpoint and proved that the eigenvalues of particular index for an infinite set of different boundary conditions at the right endpoint uniquely determine this equation.

Our aim in this note is to give a very short proof of this result. We will formulate our theorem for more general boundary conditions at the left endpoint, as this does not require any substantial changes. One could also consider more general boundary conditions at the right endpoint, but that would require some restrictions on them (see the remarks at the end).

As in the original proof, we also proceed by reducing the claim to the two-spectra uniqueness result. To this end, McLaughlin and Rundell introduce an auxiliary analytic function and use the identity theorem for analytic functions. We instead directly apply the identity theorem for meromorphic functions to the logarithmic derivative of a solution.

Denote by $\boldsymbol{\uplambda}_k(q, f, b)$, $k = 0, 1, 2, \dots$ the eigenvalues of the boundary value problem generated by~(\ref{eq:SL}) and the boundary conditions
\begin{equation} \label{eq:boundary}
  y'(0) = -f(\lambda) y(0), \qquad y'(\pi) = b y(\pi),
\end{equation}
where $q \in \mathscr{L}_1(0, \pi)$ is real-valued, $f$ is a rational Herglotz--Nevanlinna function, and $b \in \mathbb{R} \cup \{ \boldsymbol{\infty} \}$.
We use the symbol $\boldsymbol{\infty}$ to indicate the Dirichlet boundary condition at either endpoint. Let $\varphi(x, \lambda)$ be a solution of~(\ref{eq:SL}) satisfying the first boundary condition in~(\ref{eq:boundary}) (the fact that it is only defined up to a constant multiple by this condition is not important for us here because we only need its logarithmic derivative). Then $\boldsymbol{\uplambda}_k(q, f, \boldsymbol{\infty})$ (respectively, $\boldsymbol{\uplambda}_k(q, f, 0)$) are the poles (respectively, the zeros) of the meromorphic function $\varphi'(\pi, \lambda)/\varphi(\pi, \lambda)$, which is strictly monotone decreasing from $+\infty$ to $-\infty$ in every interval of the form $\left( \boldsymbol{\uplambda}_k(q, f, \boldsymbol{\infty}), \boldsymbol{\uplambda}_{k+1}(q, f, \boldsymbol{\infty}) \right)$, and the eigenvalues $\boldsymbol{\uplambda}_k(q, f, b)$ of~(\ref{eq:SL})-(\ref{eq:boundary}) are exactly the roots of the equation $\varphi'(\pi, \lambda)/\varphi(\pi, \lambda) = b$ \cite[the proof of Lemma 2.3]{G17}.

In order to state our two-spectra uniqueness result, consider for a moment the symmetric continuation of our boundary value problem to the interval $(0, 2\pi)$, i.e. the one obtained by setting $q(x) := q(2\pi-x)$ for $x \in (\pi, 2\pi)$ and replacing the second boundary condition in~(\ref{eq:boundary}) by $y'(2\pi) = f(\lambda) y(2\pi)$. Then the odd-indexed and even-indexed eigenvalues of this symmetric problem coincide with $\boldsymbol{\uplambda}_k(q, f, \boldsymbol{\infty})$ and $\boldsymbol{\uplambda}_k(q, f, 0)$, respectively, and the uniqueness part of \cite[Theorem 4.9]{G17} says that these eigenvalues uniquely determine its coefficients. Therefore we immediately have
\begin{lemma}
  If $\boldsymbol{\uplambda}_k(q_1, f_1, \boldsymbol{\infty}) = \boldsymbol{\uplambda}_k(q_2, f_2, \boldsymbol{\infty})$ and $\boldsymbol{\uplambda}_k(q_1, f_1, 0) = \boldsymbol{\uplambda}_k(q_2, f_2, 0)$ for $k = 0, 1, 2, \dots$, then $q_1(x) = q_2(x)$ a.e. on $(0, \pi)$ and $f_1 = f_2$.
\end{lemma}

With all these preparations, we can now prove our main result.

\begin{theorem}
  If $\boldsymbol{\uplambda}_k(q_1, f_1, b_j) = \boldsymbol{\uplambda}_k(q_2, f_2, b_j)$ for a fixed index $k \ge 1$ and an infinite number of distinct $b_j$, then $q_1(x) = q_2(x)$ a.e. on $(0, \pi)$ and $f_1 = f_2$.
\end{theorem}
\begin{proof}
  As the difference of the meromorphic functions $\varphi'_1(\pi, \lambda)/\varphi_1(\pi, \lambda)$ and $\varphi'_2(\pi, \lambda)/\varphi_2(\pi, \lambda)$ corresponding to these two problems equals zero at infinitely many points $\boldsymbol{\uplambda}_k(q_1, f_1, b_j)$ from the bounded interval $\left( \boldsymbol{\uplambda}_{k-1}(q_1, f_1, \boldsymbol{\infty}), \boldsymbol{\uplambda}_k(q_1, f_1, \boldsymbol{\infty}) \right)$, by the identity theorem for meromorphic functions these two functions---and hence their zeros and poles---coincide, and the claim immediately follows from the lemma.
\end{proof}

This theorem holds for the lowest index $k = 0$ too, if the set of possible $b_j$'s is bounded from above (cf.~\cite[Remark]{MR87}). Moreover, the theorem also holds in the case of eigenparameter-dependent boundary conditions with rational Herglotz--Nevanlinna functions $b_j(\lambda)$ at the right endpoint, under the assumption that the set of their poles is bounded from below. Finally, using \cite[Theorem 5.4]{G19a} or \cite[Theorem 4.8]{G20} instead of \cite[Theorem 4.9]{G17} above, it is straightforward to generalize this result to distributional potentials and to inverse square singularities at the left endpoint.


\begin{thebibliography}{9}

\bibitem{G17} N. J. Guliyev,
\emph{Essentially isospectral transformations and their applications},
Ann. Mat. Pura Appl. (4) \textbf{199} (2020), no. 4, 1621--1648.
\href{https://arxiv.org/abs/1708.07497}{arXiv:1708.07497}

\bibitem{G19a} N. J. Guliyev,
\emph{Schr\"{o}dinger operators with distributional potentials and boundary conditions dependent on the eigenvalue parameter},
J. Math. Phys. \textbf{60} (2019), no. 6, 063501, 23 pp.
\href{https://arxiv.org/abs/1806.10459}{arXiv:1806.10459}

\bibitem{G20} N. J. Guliyev,
\emph{Inverse square singularities and eigenparameter-dependent boundary conditions are two sides of the same coin},
Q. J. Math. \textbf{74} (2023), no. 3, 889--910.
\href{https://arxiv.org/abs/2001.00061}{arXiv:2001.00061}

\bibitem{MR87} J. R. McLaughlin and W. Rundell,
\emph{A uniqueness theorem for an inverse Sturm--Liouville problem},
J. Math. Phys. \textbf{28} (1987), no. 7, 1471--1472.

\end{thebibliography}
\end{document}